\documentclass{gtpart}
\usepackage{amsmath,amssymb,amsthm,stmaryrd}
\usepackage[all]{xy}
\usepackage{tikz}
\usepackage{url}
\usepackage{hyperref}
\usepackage{enumerate}
\usepackage{tensor}
\usepackage{mathrsfs}
\usepackage{graphicx}
\usepackage{mathtools}

\title{The power operation structure on Morava $E$-theory of height 2 at the prime 3}
\author{Yifei Zhu}
\givenname{Yifei}
\surname{Zhu}
\address{Department of Mathematics\\University of Minnesota\\Minneapolis, MN 55455\\USA}
\email{zyf@math.umn.edu}


\newtheorem{thm}{Theorem}
\newtheorem{cor}[thm]{Corollary}
\newtheorem{prop}[thm]{Proposition}

\theoremstyle{definition}
\newtheorem{defn}[thm]{Definition}
\theoremstyle{remark}
\newtheorem{rmk}[thm]{Remark}
\newtheorem{ex}[thm]{Example}

\def\co{\colon\thinspace}
\newcommand{\mb}[1]{\mathbb{#1}}

\newcommand{\Spec}{{\rm Spec\thinspace}}
\newcommand{\Proj}{{\rm Proj\thinspace}}

\newcommand{\Mod}{{\rm Mod}}
\newcommand{\Alg}{{\rm Alg}}

\newcommand{\DL}{Dyer-Lashof~}
\newcommand{\BF}{{\mb F}}
\newcommand{\BG}{{\mb G}}
\newcommand{\BP}{{\mb P}}
\newcommand{\BZ}{{\mb Z}}
\newcommand{\HC}{\widehat{C}}
\newcommand{\HS}{\widehat{S}}
\newcommand{\Tf}{\widetilde{f}}
\newcommand{\Tp}{\widetilde{\psi}}
\newcommand{\md}{~~{\rm mod}~}
\newcommand{\ad}{{\rm and}}
\newcommand{\A}{\alpha}
\newcommand{\G}{\Gamma}
\newcommand{\g}{\gamma}
\newcommand{\K}{\kappa}
\newcommand{\p}{\psi^3}
\newcommand{\s}{S^\bullet}
\newcommand{\isog}[1]{Proposition \ref{prop:isog}\thinspace \eqref{isog(#1)}}
\newcommand{\q}[1]{Proposition \ref{prop:Q}\thinspace \eqref{Q(#1)}}
\newcommand{\go}[1]{Definition \ref{def:go}\thinspace \eqref{go(#1)}}

\begin{document}
\begin{abstract}
 We give explicit calculations of the algebraic theory of power 
 operations for a specific Morava $E$-theory spectrum and its 
 $K(1)$-localization.  These power operations arise from the universal 
 degree-3 isogeny of elliptic curves associated to the $E$-theory.  
\end{abstract}

\maketitle
\section{Introduction}

Suppose $E$ is a commutative $S$-algebra, in the sense of \cite{EKMM}, 
and $A$ is a commutative $E$-algebra.  We want to capture the properties 
and underlying structure of the homotopy groups $\pi_* A = A_*$ of $A$, 
by studying operations associated to the cohomology theory that $E$ 
represents.  

An important family of cohomology operations, called {\em power 
operations}, is constructed via the extended powers.  Specifically, 
consider the {\em $m$'th extended power} functor 
\[
 \BP_E^m (-) \coloneqq (-)^{\wedge_E m} / \Sigma_m \co \Mod_E \to \Mod_E 
\]
on the category of $E$-modules, which sends an $E$-module to its 
$m$-fold smash product over $E$ modulo the action by the symmetric group 
on $m$ letters.  The $\BP_E^m (-)$'s assemble together to give the {\em 
free commutative $E$-algebra} functor 
\[
 \BP_E (-) \coloneqq \bigvee_{m \geq 0} \BP_E^m (-) \co \Mod_E \to \Alg_E 
\]
from the category of $E$-modules to the category of commutative 
$E$-algebras.  These functors descend to homotopy categories.  In 
particular, each $\A \in \pi_{d+i}~\BP_E^m (\Sigma^d E)$ gives rise to a 
power operation 
\[
 Q_\A \co A_d \to A_{d+i} 
\]
(cf.~\cite[Sections I.2 and IX.1]{H_infty} and \cite[Section 3]{cong}).  

Under the action of power operations, $A_*$ is an algebra over some 
operad in $E_*$-modules involving the structure of $E_* B\Sigma_m$ for 
all $m$.  This operad is traditionally called a {\em \DL algebra}, or 
more precisely, a \DL {\em theory} as the {\em algebraic theory} of 
power operations acting on the homotopy groups of commutative 
$E$-algebras (cf.~\cite[Chapters III, VIII, and IX]{H_infty} and 
\cite[Section 9]{lpo}).  

A specific case is when $E$ represents a Morava $E$-theory of height $n$ 
and $A$ is $K(n)$-local.  Morava $E$-theory spectra play a crucial role 
in modern stable homotopy theory, particularly in the work of Ando, 
Hopkins, and Strickland on the topological approach to elliptic genera 
(see \cite{cube}).  As recalled in \cite[1.5]{cong}, the $K(n)$-local 
$E$-\DL theory is largely understood based on work of those authors.  In 
\cite{cong}, Rezk maps out the foundations of this theory.  He gives a 
congruence criterion for an algebra over the \DL theory 
(\cite[Theorem A]{cong}).  This enables one to study the \DL 
{\em theory}, which models all the algebraic structure naturally 
adhering to $A_*$, by working with a certain associative ring $\G$ as 
the \DL {\em algebra}.  Moreover, Rezk provides a geometric description 
of this congruence criterion, in terms of sheaves on the moduli problem 
of deformations of formal groups and Frobenius isogenies (see 
\cite[Theorem B]{cong}).  This connects the structure of $\G$ to the 
geometry underlying $E$, moving one step forward from a workable object 
$\G$ to things that are computable.  In a companion paper \cite{h2p2}, 
Rezk gives explicit calculations of the \DL theory for a specific Morava 
$E$-theory of height $n = 2$ at the prime 2.  

The purpose of this paper is to make available calculations analogous to 
some of the results in \cite{h2p2}, at the prime 3, together with 
calculations of the corresponding power operations on the 
$K(1)$-localization of the Morava $E$-theory spectrum.

\subsection{Outline of the paper}

As in \cite{h2p2}, the computation of power operations in this paper 
follows the approach of \cite{steenrod}: one first defines a total power 
operation, and then uses the computation of the cohomology of the 
classifying space $B\Sigma_m$ for the symmetric group to obtain 
individual power operations.  These two steps are carried out in 
Sections \ref{sec:total} and \ref{sec:individual} respectively.  

In Section \ref{sec:total}, by doing calculations with elliptic curves 
associated to our Morava $E$-theory $E$, we give formulas for the total 
power operation $\p$ on $E_0$ and the ring $S_3$ which represents the 
corresponding moduli problem.  

In Section \ref{sec:individual}, based on calculations of 
$E^* B\Sigma_m$ in \cite{Str98} as reflected in the formula for $S_3$, 
we define individual power operations, and derive the relations they 
satisfy.  In view of the general structures studied in \cite{cong}, we 
then get an explicit description of the \DL algebra $\G$ for 
$K(2)$-local commutative $E$-algebras.  

In Section \ref{sec:K(1)}, we describe the relationship between the 
total power operation $\p$, at height 2, and the corresponding 
$K(1)$-local power operations.  We then derive formulas for the latter 
from the calculations in Section \ref{sec:total}.  

\begin{rmk}
\label{rmk:grading}
 In Section \ref{sec:total}, we do calculations with a universal 
 elliptic curve over {\em all} of the moduli stack which is an affine 
 open subscheme of a weighted projective space (cf.~Proposition 
 \ref{prop:C}).  At the prime 3, the supersingular locus consists of a 
 single closed point, and the corresponding Morava $E$-theory arises 
 {\em locally} in an affine coordinate chart of this weighted projective 
 space containing the supersingular locus.  In this paper we choose a 
 particular affine coordinate chart for computing the homotopy groups of 
 the $E$-theory spectrum and the power operations; we hope that the 
 generality of the calculations in Section \ref{sec:total} makes it 
 easier to work with other coordinate charts as well.  
\end{rmk}

\begin{rmk}
\label{rmk:parameter}
 The ring $S_3$ turns out to be an algebra with one generator over the 
 base ring where our elliptic curve is defined (cf.~\isog{i} and 
 \eqref{S_3}).  This generator appears as a parameter in the formulas 
 for the total power operation $\p$, and is responsible for how the 
 individual power operations are defined and how their formulas look.  
 Different choices of this parameter result in different bases of the 
 \DL algebra $\G$.  The parameter in this paper comes from the relative 
 cotangent space of the elliptic curve at the identity (see \isog{iv}, 
 Corollary \ref{cor:K'}, and Remark \ref{rmk:K'}).  This choice is 
 convenient for deriving Adem relations in \q{iv}, and it fits into the 
 treatment of gradings in \cite[Section 2]{cong} (see \go{ii} and 
 Theorem \ref{thm:gamma}).  We should point out that our choice is by no 
 means canonical.  We do not know yet, as part of the structure of the 
 \DL algebra, if there is a canonical basis which is both geometrically 
 interesting and computationally convenient.  Somewhat surprisingly, 
 although it appears to come from different considerations, our choice 
 has an analog at the prime 2 which coincides with the parameter used in 
 \cite{h2p2} (see Remarks \ref{rmk:K} and \ref{rmk:KK'}).  The 
 calculations follow a recipe in hope of generalizing to other Morava 
 $E$-theories of height 2; we hope to address these matters and 
 recognize more of the general patterns based on further computational 
 evidence.  
\end{rmk}

\subsection{Acknowledgements}

I thank Charles Rezk for his encouragement on this work, and for his 
observation in a correspondence which led to Proposition 
\ref{prop:frob^2} and Corollary \ref{cor:K'}.  

I thank Kyle Ormsby for helpful discussions on Section \ref{sec:total}, 
and for directing me to places in the literature.  

I thank Tyler Lawson for the sustained support from him I received as a 
student.

\subsection{Conventions}

Let $p$ be a prime, $q$ a power of $p$, and $n$ a positive integer.  We 
use the symbols 
\[
 \BF_q\text{,}~~\BZ_q\text{,}~~\ad~~\BZ/n 
\]
to denote a field with $q$ elements, the ring of $p$-typical Witt 
vectors over $\BF_q$, and the additive group of integers modulo $n$, 
respectively.  

If $R$ is a ring, then $R\llbracket x \rrbracket$ and $R (\!(x)\!)$ 
denote the rings of formal power series and formal Laurent series over 
$R$ in the variable $x$ respectively.  If $I \subset R$ is an ideal, 
then $R_I^\wedge$ denotes the completion of $R$ with respect to $I$.  

If $E$ is an elliptic curve and $m$ is an integer, then $[m]$ denotes 
the multiplication-by-$m$ map on $E$, and $E[m]$ denotes the $m$-torsion 
subgroup scheme of $E$.  

All formal groups mentioned in this paper will be commutative and 
one-dimensional.  

The terminology for the structure of the \DL theory will follow 
\cite{cong} and \cite{h2p2}; some of the notions there are taken in turn 
from \cite{BW} and \cite{V}.

\section{Total power operations}
\label{sec:total}

\subsection{A universal elliptic curve and a Morava $E$-theory spectrum}
\label{subsec:ec}

A Morava $E$-theory of height 2 at the prime 3 has its formal group as 
the universal deformation of a height-2 formal group over a perfect 
field of characteristic 3.  Given a supersingular elliptic curve over 
such a field, its formal completion at the identity produces a formal 
group of height 2.  To study power operations for the corresponding 
$E$-theory, we do calculations with the universal deformation of that 
supersingular elliptic curve which is a family of elliptic curves with a 
$\G_1(N)$-structure (see \cite[Section 3.2]{KM}) where $N$ is prime to 3.  
Here is a specific model (cf.~\cite[4(4.6a)]{husemoller}).  

\begin{prop}
\label{prop:C}
 Over $\BZ [1/4]$, the moduli problem of nonsingular elliptic curves 
 with a choice of a point of exact order 4 and a nowhere-vanishing 
 invariant one-form is represented by 
 \begin{equation}
 \label{Cxy}
  C \co y^2 + a x y + a b y = x^3 + b x^2 
 \end{equation}
 with chosen point $(0,0)$ and one-form 
 $dx / (2 y + a x + a b) = dy / (3 x^2 + 2 b x - a y)$ over the graded 
 ring 
 \[
  \s \coloneqq \BZ [1/4] [a, b, \Delta^{-1}] 
 \]
 where $|a| = 1$, $|b| = 2$, and $\Delta = a^2 b^4 (a^2 - 16 b)$.  
\end{prop}
\begin{proof}
 Let $P$ be the chosen point of exact order 4.  Since $2P$ is 2-torsion, 
 the tangent line of the elliptic curve at $P$ passes through $2P$, and 
 the tangent line at $2P$ passes through the identity at the infinity.  
 With this observation, the rest of the proof is analogous to that of 
 \cite[Proposition 3.2]{tmf3}.  
\end{proof}

Over a finite field of characteristic 3, $C$ is supersingular precisely 
when the quantity 
\begin{equation}
\label{H}
 H \coloneqq a^2 + b 
\end{equation}
vanishes (cf.~\cite[V.4.1a]{AEC}).  As $(3,H)$ is a homogeneous maximal 
ideal of $\s$ corresponding to the closed subscheme $\Spec \BF_3$, the 
supersingular locus consists of a single closed point, and $C$ restricts 
to $\BF_3$ as 
\[
 C_0 \co y^2 + x y - y = x^3 - x^2.  
\]

From the above universal deformation $C$ of $C_0$, we next produce a 
Morava $E$-theory spectrum which is 2-periodic.  We follow the 
convention that elements in algebraic degree $n$ lie in topological 
degree $2n$, and work in an affine \'etale coordinate chart of the 
weighted projective space $\Proj \BZ [1/4] [a, b]$ (see Remark 
\ref{rmk:grading}).  Define elements $u$ and $c$ such that 
\[
 a = u c \qquad \ad \qquad b = u^2.  
\]
Consider the graded ring 
\[
 \s [u^{-1}] \cong \BZ [1/4] [a, \Delta^{-1}] [u^{\pm1}] 
\]
where $|u| = 1$, and denote by $S$ its subring of elements in degree 0 
so that 
\begin{equation}
\label{S}
 S \cong \BZ [1/4] [c, \delta^{-1}] 
\end{equation}
where $\delta = u^{-12} \Delta = c^2 (c^2 - 16)$.  Write 
\[
 \HS = \BZ_9 \llbracket h \rrbracket 
\]
where 
\begin{equation}
\label{h}
 h \coloneqq u^{-2} H = c^2 + 1.  
\end{equation}
Let $i$ be an element generating $\BZ_9$ over $\BZ_3$ with $i^2 = -1$.  
We may choose 
\[
 c \equiv i \md (3,h) 
\]
and we have 
\[
 \delta \equiv -1 \md (3,h) 
\]
where $(3,h)$ is the maximal ideal of the complete local ring $\HS$.  
Then by Hensel's lemma, both $c$ and $\delta$ lie in $\HS$, and both are 
invertible.  Thus 
\[
 \HS \cong S_{(3,h)}^\wedge.  
\]

Now $C$ restricts to $S$ as 
\begin{equation}
\label{Cc}
 y^2 + c x y + c y = x^3 + x^2.  
\end{equation}
Let $\HC$ be the formal completion of $C$ over $S$ at the identity.  It 
is a formal group over $\HS$, and its reduction to $\BF_9 = \HS / (3,h)$ 
is a formal group $\BG$ of height 2 in view of \eqref{h} and \eqref{H}.  
By the Serre-Tate theorem (see \cite[2.9.1]{KM}), 3-adically the 
deformation theory of an elliptic curve is equivalent to the deformation 
theory of its 3-divisible group, and thus $\HC$ is the universal 
deformation of $\BG$ in view of Proposition \ref{prop:C}.  Let $E$ be 
the $E_\infty$-ring spectrum which represents the Morava $E$-theory 
associated to $\BG$ (see \cite[Corollary 7.6]{GH}).  Then 
\[
 E_* \cong \BZ_9 \llbracket h \rrbracket [u^{\pm 1}] 
\]
where $u$ is in topological degree 2, and it corresponds to a local 
uniformizer at the identity of $C$.

\subsection{Points of exact order 3}

To study $C$ in a formal neighborhood of the identity, it is convenient 
to make a change of variables.  Let 
\[
 u = \frac{x}{y} \quad \ad \quad v = \frac{1}{y}, \qquad {\rm so} \qquad x = \frac{u}{v} \quad \ad \quad y = \frac{1}{v}.  
\]
The identity of $C$ is then $(u,v) = (0,0)$, with $u$ a local 
uniformizer.  The equation \eqref{Cxy} of $C$ becomes 
\begin{equation}
\label{Cuv}
 v + a u v + a b v^2 = u^3 + b u^2 v.  
\end{equation}

\begin{prop}
\label{prop:tors}
 On the elliptic curve $C$ over $\s$, the $uv$-coordinates $(d,e)$ of 
 any nonzero 3-torsion point satisfy the identities 
 \begin{equation}
 \label{f}
  f(d) = 0 
 \end{equation}
 and 
 \begin{equation}
 \label{g}
  e = g(d) 
 \end{equation}
 where $f, g \in \s [u]$ are given by 
 \begin{equation*}
 \begin{split}
  f(u) = & ~ b^4 u^8 + 3 a b^3 u^7 + 3 a^2 b^2 u^6 + (a^3 b + 7 a b^2) u^5 + (6 a^2 b - 6 b^2) u^4 + 9 a b u^3 \\
         & + (-a^2 + 8 b) u^2 - 3 a u - 3, \\
  g(u) = & -\frac{1}{a (a^2 - 16 b)} \big( a b^3 u^7 + (3 a^2 b^2 - 2 b^3) u^6 + (3 a^3 b -6 a b^2) u^5 + (a^4 + a^2 b \\
         & + 2 b^2) u^4 + (4 a^3 - 15 a b) u^3 + 18 b u^2 - 12 a u - 18 \big).  
 \end{split}
 \end{equation*}
\end{prop}
\begin{proof}
 \footnote{See Appendix \ref{apx:tors} for explicit formulas for the 
 polynomials $\Tf$, $Q_1$, $R_1$, $Q_2$, $R_2$, $K$, $L$, $M$, and $N$ 
 that appear in the proof.  }
 Given the elliptic curve $C$ with equation \eqref{Cxy}, a nonzero point 
 $Q$ is 3-torsion if and only if the polynomial 
 \[
  \psi_3 (x) \coloneqq 3 x^4 + (a^2 + 4 b) x^3 + 3 a^2 b x^2 + 3 a^2 b^2 x + a^2 b^3 
 \]
 vanishes at $Q$ (cf.~\cite[Exercise 3.7f]{AEC}).  Substituting 
 $x = u / v$ and clearing the denominators, we get a polynomial 
 \[
  \Tp_3(u,v) \coloneqq 3 u^4 + (a^2 + 4 b) u^3 v + 3 a^2 b u^2 v^2 + 3 a^2 b^2 u v^3 + a^2 b^3 v^4.  
 \]
 As $Q = (d,e)$ in $uv$-coordinates, we then have 
 \begin{equation}
 \label{Tp}
  \Tp_3(d,e) = 0.  
 \end{equation}

 To get the polynomial $f$, we take $v$ as variable and rewrite 
 \eqref{Cuv} as a quadratic equation 
 \begin{equation}
 \label{quadratic}
  a b v^2 + (-b u^2 + a u + 1) v - u^3 = 0, 
 \end{equation}
 where the leading coefficient $a b$ is invertible in 
 $\s = \BZ [1/4] [a, b, \Delta^{-1}]$ as $\Delta = a^2 b^4 (a^2 - 16 b)$.  
 Define 
 \begin{equation}
 \label{Tfdef}
  \Tf(u) \coloneqq \Tp_3(u,v) \Tp_3(u,\bar{v}) 
 \end{equation}
 where $v$ and $\bar{v}$ are formally the conjugate roots of 
 \eqref{quadratic} so that we compute $\Tf$ in terms of $u$ by 
 substituting 
 \[
  v + \bar{v} = \frac{b u^2 - a u - 1}{a b} \qquad \ad \qquad v \bar{v} = -\frac{u^3}{a b}.  
 \]
 We then factor $\Tf$ over $\s$ as 
 \begin{equation}
 \label{Tffactor}
  \Tf(u) = -\frac{u^4 f(u)}{a^2 b} 
 \end{equation}
 with $f$ the stated polynomial of order 8.  We check that $f$ is 
 irreducible by applying Eisenstein's criterion to the homogeneous prime 
 ideal $(3,H)$ of $\s$.  

 We have $\Tf(d) = 0$ by \eqref{Tfdef} and \eqref{Tp}.  To see 
 $f(d) = 0$, consider the closed subscheme $D \subset C[3]$ of points of 
 exact order 3.  By \cite[2.3.1]{KM} it is finite locally free of rank 8 
 over $\s$.  By the Cayley-Hamilton theorem, as a global section of $D$, 
 $u$ locally satisfies a homogeneous monic equation of order 8, and this 
 equation locally defines the rank-8 scheme $D$.  Since $D$ is affine, 
 it is then globally defined by such an equation.  In view of 
 $\Tf(d) = 0$ and \eqref{Tffactor}, we determine this equation, and (up 
 to a unit in $\s$) get the first stated identity \eqref{f}.  

 To get the polynomial $g$, we note that both the quartic polynomial 
 \[
  A(v) \coloneqq \Tp_3(d,v) 
 \]
 and the quadratic polynomial 
 \[
  B(v) \coloneqq a b v^2 + (-b d^2 + a d + 1) v - d^3 
 \]
 vanish at $e$, and thus so does their greatest common divisor (gcd).  
 Applying the Euclidean algorithm (see Appendix \ref{apx:tors} for 
 explicit expressions), we have 
 \begin{equation*}
 \begin{split}
  A(v) = & ~ Q_1(v) B(v) + R_1(v), \\
  B(v) = & ~ Q_2(v) R_1(v) + R_2, 
 \end{split}
 \end{equation*}
 where 
 \[
  R_1(v) = K(d) v + L(d) 
 \]
 for some polynomials $K$ and $L$, and $R_2 = 0$ in view of \eqref{f}.  
 Thus $R_1(v)$ is the gcd of $A(v)$ and $B(v)$, and hence 
 \[
  K(d) e + L(d) = R_1(e) = 0.  
 \]
 To write $e$ in terms of $d$ from the above identity, we apply the 
 Euclidean algorithm to $f$ and $K$.  Their gcd turns out to be 1, and 
 thus there are polynomials $M$ and $N$ with 
 \[
  M(u) f(u) + N(u) K(u) = 1.  
 \]
 By \eqref{f} we then have $N(d) K(d) = 1$, and thus 
 \[
  e = -N(d) L(d) = g(d) 
 \]
 where $g$ is as stated.  
\end{proof}

\begin{rmk}
\label{rmk:dmod3}
 The formula for $f$ in Proposition \ref{prop:tors} satisfies a 
 congruence 
 \[
  f(u) \equiv u^2 (b^4 u^6 + a b H u^3 - H) \md 3.  
 \]
 The two roots (counted with multiplicity) of $f(u)$ which reduce to 
 zero modulo 3 correspond to the two nonzero points in the unique 
 order-3 subgroup of $C$ in a formal neighborhood of the identity.  
\end{rmk}

\subsection{A universal isogeny and a total power operation}

\begin{prop}
\label{prop:isog}
 \mbox{}
 \begin{enumerate}[(i)]
  \item \label{isog(i)} The universal degree-3 isogeny $\psi$ with 
  source $C$ is defined over the graded ring 
  \[
   \s_3 \coloneqq \s [\K] \big/ \big( W(\K) \big) 
  \]
  where $|\K| = -2$ and 
  \begin{equation}
  \label{W}
   W(\K) = \K^4 - \frac{6}{b^2} ~ \K^2 + \frac{a^2 - 8 b}{b^4} ~ \K - \frac{3}{b^4}, 
  \end{equation}
  and has target the elliptic curve 
  \[
   C' \co v + a' u v + a' b' v^2 = u^3 + b' u^2 v 
  \]
  where 
  \begin{equation*}
  \begin{split}
   a' = & ~ \frac{1}{a} \big( (a^2 b^4 - 4 b^5) \K^3 + 4 b^4 \K^2 + (-6 a^2 b^2 + 20 b^3) \K + a^4 - 12 a^2 b + 12 b^2 \big), \\
   b' = & ~ b^3.  
  \end{split}
  \end{equation*}

  \item \label{isog(ii)} The kernel of $\psi$ is generated by a point 
  $Q$ of exact order 3 with coordinates $(d,e)$ satisfying 
  \begin{equation}
  \label{K}
  \begin{split}
   \K = & -\frac{1}{a^2 - 16 b} \big( a b^3 d^7 + (3 a^2 b^2 - 2 b^3) d^6 + (3 a^3 b - 6 a b^2) d^5 + (a^4 \\
        & + a^2 b + 2 b^2) d^4 + (4 a^3 - 15 a b) d^3 + (a^2 + 2 b) d^2 - 12 a d - 18 \big) \\
      = & ~ a e - d^2.  
  \end{split}
  \end{equation}

  \item \label{isog(iii)} The restriction of $\psi$ to the supersingular 
  locus at the prime 3 is the 3-power Frobenius endomorphism.  

  \item \label{isog(iv)} The induced map $\psi^*$ on the relative 
  cotangent space of $C'$ at the identity sends $du$ to $\K du$.  
 \end{enumerate}
\end{prop}
\begin{proof}
 \footnote{See Appendix \ref{apx:isog} for the power series expansion of 
 $v$ and details of the calculations involving the group law on $C$ that 
 appear in the proof.  }
 Let $P = (u,v)$ be a point on $C$, and $Q = (d,e)$ be a nonzero 
 3-torsion point.  Rewriting \eqref{Cuv} as 
 \[
  v = u^3 + b u^2 v - a u v - a b v^2, 
 \]
 we express $v$ as a power series in $u$ by substituting this equation 
 into itself recursively.  For the purpose of our calculations, we take 
 this power series up to $u^{12}$ as an expression for $v$, and write 
 $e = g(d)$ as in \eqref{g}.  

 Define functions $u'$ and $v'$ by 
 \begin{equation}
 \label{u'v'}
 \begin{split}
  u' \coloneqq & ~ u(P) \cdot u(P-Q) \cdot u(P+Q), \\
  v' \coloneqq & ~ v(P) \cdot v(P-Q) \cdot v(P+Q), 
 \end{split}
 \end{equation}
 where $u(-)$ and $v(-)$ denote the $u$-coordinate and $v$-coordinate of 
 a point respectively.  By computing the group law on $C$, we express 
 $u'$ and $v'$ as power series in $u$: 
 \begin{equation}
 \label{KL}
 \begin{split}
  u' = & ~ \K u + (\text{higher-order terms}), \\
  v' = & ~ \lambda u^3 + (\text{higher-order terms}), 
 \end{split}
 \end{equation}
 where the coefficients ($\K$, $\lambda$, etc.)~involve $a$, $b$, and 
 $d$.  In particular, in view of \eqref{f}, we compute that $\K$ 
 satisfies $W(\K) = 0$ with $|\K| = -2$ as stated in \eqref{isog(i)}.  

 Now define the isogeny $\psi \co C \to C'$ by 
 \begin{equation}
 \label{psi}
  u\big( \psi(P) \big) \coloneqq u' \qquad \ad \qquad v\big( \psi(P) \big) \coloneqq \frac{\K^3}{\lambda} \cdot v', 
 \end{equation}
 where we introduce the factor $\K^3 / \lambda$ so that the equation of 
 $C'$ will be in the Weierstrass form.  Using \eqref{KL} (see Appendix 
 \ref{apx:isog} for explicit expressions), we then determine the 
 coefficients in a Weierstrass equation and get the stated equation of 
 $C'$.  

 We next check the statement of \eqref{isog(ii)}.  In view of 
 \eqref{psi} and \eqref{u'v'}, the kernel of $\psi$ is the order-3 
 subgroup generated by $Q$.  In \eqref{K}, the first identity is 
 computed in \eqref{KL}; we then compare it with the formula for $g$ in 
 Proposition \ref{prop:tors} and get the second identity.  

 For \eqref{isog(iii)}, recall from Section \ref{subsec:ec} that the 
 supersingular locus at the prime 3 is $\Spec \BF_3$.  Over $\BF_3$, 
 since $C[3] = 0$ by \cite[V.3.1a]{AEC}, $Q$ coincides with the identity, 
 and thus 
 \[
  u\big( \psi(P) \big) = u(P) \cdot u(P-Q) \cdot u(P+Q) = \big( u(P) \big)^3.  
 \]
 As the $u$-coordinate is a local uniformizer at the identity, $\psi$ 
 then restricts to $\BF_3$ as the 3-power Frobenius endomorphism.  

 The statement of \eqref{isog(iv)} follows by definition of $\K$ in 
 \eqref{KL}.  
\end{proof}

\begin{rmk}
 In view of \isog{iii}, the formal completion of $\psi \co C \to C'$ at 
 the identity of $C$ is a {\em deformation of Frobenius} in the sense of 
 \cite[11.3]{cong}.  When it is clear from the context, we will simply 
 call $\psi$ itself a deformation of Frobenius.  
\end{rmk}

\begin{rmk}
\label{rmk:K}
 From \eqref{u'v'} and \eqref{KL} we have 
 \begin{equation}
 \label{norm}
  u(P-Q) \cdot u(P+Q) = \K + u \cdot (\text{higher-order terms}).  
 \end{equation}
 In particular $u(-Q) \cdot u(Q) = \K$ 
 (cf.~\cite[Proposition 7.5.2 and Section 7.7]{KM}).  The analog of $\K$ 
 at the prime 2 coincides with $d$ as studied in \cite[Section 3]{h2p2}.  
\end{rmk}

Recall from Section \ref{subsec:ec} that 
\[
 E^0 \cong \BZ_9 \llbracket h \rrbracket = \HS \cong S_{(3,h)}^\wedge 
\]
in which $c$ and $i$ are elements with $c^2 + 1 = h$ and $i^2 = -1$.  
Given the graded ring $\s_3$ in \isog{i}, define 
\begin{equation}
\label{S_3}
 S_3 \coloneqq S [\A] / \big( w(\A) \big) 
\end{equation}
where 
\begin{equation}
\label{w}
 w(\A) = \A^4 - 6 \A^2 + (c^2 - 8) \A - 3 
\end{equation}
(cf.~the definition of $S$ from $\s$ in \eqref{S}).  By 
\cite[Theorem 1.1]{Str98} we have 
\[
 E^0 B\Sigma_3 / I \cong \big( S_3 \big)_{(3,h)}^\wedge 
\]
where 
\begin{equation}
\label{transfer}
 I \coloneqq \bigoplus_{0<i<3} {\rm image} \big( E^0 B(\Sigma_i \times \Sigma_{3-i}) \xrightarrow{\rm transfer} E^0 B\Sigma_3 \big) 
\end{equation}
is the {\em transfer ideal}.  
In view of this and the construction of {\em total power operations} for 
Morava $E$-theories in \cite[3.23]{cong}, we have the following 
corollary.  
\begin{cor}
\label{cor:psi3}
 The total power operation 
 \[
  \p \co E^0 \to E^0 B\Sigma_3 / I \cong E^0 [\A] \big/ \big( w(\A) \big) 
 \]
 is given by 
 \begin{equation*}
 \begin{split}
  \p(h) = & ~ h^3 + (\A^3 - 6 \A - 27) h^2 + 3 (-6 \A^3 + \A^2 + 36 \A + 67) h \\
          & + 57 \A^3 - 27 \A^2 - 334 \A - 342, \\
  \p(c) = & ~ c^3 + (\A^3 - 6 \A - 12) c - 4 (\A + 1)^2 (\A - 3) c^{-1}, \\
  \p(i) \thinspace = & -i, 
 \end{split}
 \end{equation*}
 where 
 \begin{equation}
 \label{Amod3}
  \A \equiv 0 \md 3.  
 \end{equation}
\end{cor}
\begin{proof}
 By \isog{i}, in $xy$-coordinates, $C'$ restricts to $S_3$ as 
 \[
  y^2 + c' x y + c' y = x^3 + x^2 
 \]
 where 
 \[
  c' = \frac{1}{c} \big( (c^2 - 4) \A^3 + 4 \A^2 + (-6 c^2 + 20) \A + c^4 - 12 c^2 + 12 \big).  
 \]
 By \cite[Theorem B]{cong}, since the above equation is in the form of 
 \eqref{Cc}, there is a correspondence between the restriction to $S_3$ 
 of the universal isogeny $\psi$, which is a deformation of Frobenius, 
 and the total power operation $\p$.  In particular $\p(c)$ is given by 
 $c'$.  As $\p$ is a ring homomorphism, we then get the formula for 
 $\p(h) = \p(c^2 + 1)$.  We also have 
 \[
  \big( \p(i) \big)^2 = \p(-1) = -1, 
 \]
 and thus $\p(i) = i$ or $-i$.  We exclude the former possibility in 
 view of the congruence 
 \[
  \p(i) \equiv i^3 \md 3 
 \]
 by \cite[Propositions 3.25 and 10.5]{cong}.  

 The congruence \eqref{Amod3} follows from Remark \ref{rmk:dmod3} and 
 \eqref{K}.  
\end{proof}

\section{Individual power operations}
\label{sec:individual}

\subsection{A composite of deformations of Frobenius}

Recall from Proposition \ref{prop:isog} that over $\s_3$ we have the 
universal degree-3 isogeny $\psi \co C \to C' = C/G$ where $G$ is an 
order-3 subgroup of $C$; in particular, $\psi$ is a deformation of the 
3-power Frobenius endomorphism over the supersingular locus.  We want to 
construct a similar isogeny $\psi'$ with source $C'$ so that the 
composite $\psi' \circ \psi$ will correspond to a composite of total 
power operations via \cite[Theorem B]{cong}.  

Let $G' \coloneqq C[3]/G$ which is an order-3 subgroup of $C'$.  Recall 
from Section \ref{subsec:ec} that $C$ is the universal deformation of a 
supersingular elliptic curve $C_0$.  Since the 3-divisible group of 
$C_0$ is formal, $C_0[3]$ is connected.  Thus over a formal neighborhood 
of the supersingular locus, if $G$ is the unique connected order-3 
subgroup of $C$, $G'$ is then the unique connected order-3 subgroup of 
$C'$.  As in the proof of Proposition \ref{prop:isog}, we define 
$\psi' \co C' \to C'/G'$ using a nonzero point in $G'$ (see \eqref{u'v'} 
and \eqref{psi}), and $\psi'$ is then a deformation of Frobenius.  Over 
the supersingular locus, the pair $(\psi, \psi')$ is {\em cyclic in 
standard order} in the sense of \cite[6.7.7]{KM}.  We describe it more 
precisely as below.  

\begin{prop}
\label{prop:frob^2}
 The following diagram of elliptic curves over $\s_3$ commutes: 
 \begin{equation}
 \label{frob^2}
  \begin{tikzpicture}[baseline=(current bounding box.center)]
          \node (LT) at (0, 2) {$C$}; 
          \node (MT) at (3.8, 2) {$C/G = $}; 
          \node (RT) at (4.65, 2.04) {$C'$}; 
          \node (LB) at (1.9, 0) {$C/C[3]$}; 
          \node (MB) at (3.5, 0) {$\cong \frac{C/G}{C[3]/G} = $}; 
          \node (RB) at (4.65, 0.025) {$\frac{C'}{G'}$}; 
          \node at (4.95, -0.15) {.}; 
          \draw [->] (LT) -- node [above] {$\scriptstyle \psi$} (MT); 
          \draw [->] (LT) -- node [left] {$\scriptstyle [-3]$} (LB); 
          \draw [->] (RT) -- node [right] {$\scriptstyle \psi'$} (RB); 
  \end{tikzpicture}
 \end{equation}
\end{prop}
\begin{proof}
 By \cite[2.4.2]{KM}, since $\Proj \s_3$ is connected, we need only show 
 that the locus over which $\psi' \circ \psi = [-3]$ is not empty, where 
 by abuse of notation $[-3]$ denotes the map $[-3]$ on $C$ composed with 
 the canonical isomorphism from $C/C[3]$ to $C'/G'$.  

 Recall from Section \ref{subsec:ec} that $C$ restricts to the 
 supersingular locus $\BF_3$ as 
 \[
  C_0 \co y^2 + x y - y = x^3 - x^2.  
 \]
 By \isog{iii} both $\psi$ and $\psi'$ restrict as the 3-power Frobenius 
 endomorphism $\psi_0$.  By \cite[2.6.3]{KM}, in the endomorphism ring 
 of $C_0$, $\psi_0$ is a root of the polynomial 
 \begin{equation}
 \label{charpoly}
  X^2 - {\rm trace}(\psi_0) \cdot X + 3 
 \end{equation}
 with ${\rm trace}(\psi_0)$ an integer satisfying 
 \[
  \big( {\rm trace}(\psi_0) \big)^2 \leq 12.  
 \]
 Moreover by \cite[Exercise 5.10a]{AEC}, since $C_0$ is supersingular, 
 we have 
 \[
  {\rm trace}(\psi_0) \equiv 0 \md 3.  
 \]
 Thus ${\rm trace}(\psi_0) = 0$, 3, or $-3$.  We exclude the latter two 
 possibilities by checking the action of $\psi_0$ at the 2-torsion point 
 $(1,0)$.  It then follows from \eqref{charpoly} that 
 $\psi_0 \circ \psi_0$ agrees with $[-3]$ on $C_0$ over $\BF_3$.  
\end{proof}

Analogous to \isog{iv}, let $\K'$ be the element in $\s_3$ such that 
$(\psi')^*$ sends $du$ to $\K' du$.  Note that $|\K'| = -6$.  

\begin{cor}
\label{cor:K'}
 The following relations hold in $\s_3$: 
 \[
  b^4 \K \K' + 3 = 0 
 \]
 and 
 \[
  \K' = -\K^3 + \frac{6}{b^2} ~ \K - \frac{a^2 - 8 b}{b^4}.  
 \]
\end{cor}
\begin{proof}
 The isogenies in \eqref{frob^2} induce maps on relative cotangent 
 spaces at the identity.  By \isog{iv} we then have a commutative 
 diagram 
 \begin{center}
 \begin{tikzpicture}
         \node (LT) at (0, 2) {$\K \K' du$}; 
         \node (RT) at (4.65, 2) {$\K' du$}; 
         \node (LB) at (1.9, 0) {$du$}; 
         \node (RB) at (4.65, 0) {$du$}; 
         \node at (4.95, -0.125) {.}; 
         \draw [|->] (RT) -- node [above] {$\scriptstyle \psi^*$} (LT); 
         \draw [|->] (LB) -- node [left] {$\scriptstyle [-3]^*$} (LT); 
         \draw [|->] (RB) -- node [right] {$\scriptstyle (\psi')^*$} (RT); 
         \draw [double distance=1.3pt] (LB) -- (RB); 
 \end{tikzpicture}
 \end{center}
 Thus for the first stated relation we need only show that $[3]^*$ sends 
 $du$ to $3 du / b^4$.  

 For $i = 1$, 2, 3, and 4, let $Q_i$ be a generator for each of the four 
 order-3 subgroups of $C$.  Each $Q_i$ can be chosen as $Q$ in 
 \eqref{u'v'}, and we denote the corresponding quantity $\K$ in 
 \eqref{KL} by $\K_i$.  Define an isogeny $\Psi$ with source $C$ by 
 \begin{equation*}
 \begin{split}
  u\big( \Psi(P) \big) \coloneqq & ~ u(P) \prod_{i=1}^4 \big( u(P-Q_i) \cdot u(P+Q_i) \big), \\
  v\big( \Psi(P) \big) \coloneqq & ~ v(P) \prod_{i=1}^4 \big( v(P-Q_i) \cdot v(P+Q_i) \big).  
 \end{split}
 \end{equation*}
 In view of \eqref{norm}, since $[3]$ has the same kernel as $\Psi$, we 
 have 
 \begin{equation}
 \label{s}
  [3]^* (du) = s \cdot \K_1 \K_2 \K_3 \K_4 \cdot du 
 \end{equation}
 where $s$ is a degree-0 unit in $\s$ coming from an automorphism of $C$ 
 over $\s$.  In view of \eqref{W} we have 
 \[
  \K_1 \K_2 \K_3 \K_4 = -\frac{3}{b^4}.  
 \]
 We then compute that $s = -1$ by comparing the restrictions of the two 
 sides of \eqref{s} to $S$ (see \eqref{S} for the definition of $S$): 
 $[3]^*$ becomes the multiplication-by-3 map, and $-3 / b^4$ becomes 
 $-3$ (cf.~the constant term in \eqref{w}).  Thus $[3]^*$ sends $du$ to 
 $3 du / b^4$.  

 The second stated relation follows by a computation from the first 
 relation and the relation $W(\K) = 0$ as in \isog{i}.  
\end{proof}

\begin{rmk}
\label{rmk:KK'}
 As noted in Remark \ref{rmk:K}, the (local) analog of $\K$ at the prime 
 2 coincides with the parameter $d$ in \cite[Section 3]{h2p2}.  In 
 particular, with the notations there and the equation in 
 \cite[Proposition 3.2]{tmf3}, $d$ and $d'$ satisfy an analogous 
 relation $A_3 d d' + 2 = 0$ which locally reduces to $d d' + 2 = 0$ 
 (the analog of the factor $s$ in the proof of Corollary \ref{cor:K'} 
 equals 1; cf.~\cite[Theorem 2.5.7]{andoduke}).  These arise as examples 
 of \cite[Lemma 3.21]{poonen}.  
\end{rmk}

\begin{rmk}
\label{rmk:K'}
 In view of \eqref{frob^2}, $-\psi'$ (composed with the canonical 
 isomorphism on the target) turns out to be the dual isogeny of $\psi$ 
 (cf.~the proof of \cite[2.9.4]{KM}).  If $G$ is the unique order-3 
 subgroup of $C$ in a formal neighborhood of the identity, then 
 \[
  \K \equiv 0 \md 3 
 \]
 by Remark \ref{rmk:dmod3} and \eqref{K}.  Thus in view of Corollary 
 \ref{cor:K'} and \eqref{H} we have 
 \[
  -\K' = \K^3 - \frac{6}{b^2} ~ \K + \frac{a^2 - 8 b}{b^4} \equiv \frac{H}{b^4} \md 3.  
 \]
 This congruence agrees with the interpretation of $H$ as defined by the 
 tangent map of the Verschiebung isogeny over $\BF_3$ (see 
 \cite[12.4.1]{KM}).  
\end{rmk}

\subsection{Individual power operations}

Let $A$ be a $K(2)$-local commutative $E$-algebra.  By \cite[3.23]{cong} 
and Corollary \ref{cor:psi3}, we have a total power operation 
\[
 \p \co A_0 \to A_0 \otimes_{E_0} (E^0 B\Sigma_3 / I) \cong A_0 [\A] \big/ \big( w(\A) \big).  
\]
We also have a composite of total power operations 
\begin{equation}
\label{psi3^2}
\begin{split}
 A_0 \stackrel{\p}{\longrightarrow} A_0 \otimes_{E_0} (E^0 B\Sigma_3 / I) \stackrel{\p}{\longrightarrow} 
 & ~ \big( A_0 \otimes_{E_0} (E^0 B\Sigma_3 / I) \big) \tensor[^\p]{\otimes}{_{E_0 [\A]}} (E^0 B\Sigma_3 / I) \\
 \cong \thinspace \thinspace & ~ \Big( A_0 [\A] \big/ \big( w(\A) \big) \Big) \tensor[^\p]{\otimes}{_{E_0 [\A]}} \Big( E^0 [\A] \big/ \big( w(\A) \big) \Big), 
\end{split}
\end{equation}
where the elements in the target $M \tensor[^\p]{\otimes}{_R} N$ are 
subject to the equivalence relation (as well as other ones in a usual 
tensor product) 
\[
 m \otimes (r \cdot n) \sim \big( m \cdot \p(r) \big) \otimes n 
\]
for $m \in M$, $n \in N$, and $r \in R$ with 
\[
 \p(\A) = -\A^3 + 6 \A - h + 9 
\]
in view of Corollary \ref{cor:K'}.  

\begin{defn}
 Define the {\em individual power operations} 
 \[
  Q_k \co A_0 \to A_0 
 \]
 for $k = 0$, 1, 2, and 3 by 
 \begin{equation}
 \label{Q_k}
  \p (x) = Q_0(x) + Q_1(x) \A + Q_2(x) \A^2 + Q_3(x) \A^3.  
 \end{equation}
\end{defn}

\begin{prop}
\label{prop:Q}
 The following relations hold among the individual power operations 
 $Q_0$, $Q_1$, $Q_2$, and $Q_3$: 
 \begin{enumerate}[(i)]
  \item \label{Q(i)} $Q_0(1) = 1, \quad Q_1(1) = Q_2(1) = Q_3(1) = 0;$ 

  \item \label{Q(ii)} $Q_k(x+y) = Q_k(x) + Q_k(y) \text{~for all~} k;$ 

  \item \label{Q(iii)} {\em Commutation relations }
  \begin{equation*}
  \begin{split}
   Q_0(h x) = & ~ (h^3 - 27 h^2 + 201 h - 342) Q_0(x) + (3 h^2 - 54 h + 171) Q_1(x) \qquad \qquad \\
              & + (9 h - 81) Q_2(x) + 24 Q_3(x), \\
   Q_1(h x) = & ~ (-6 h^2 + 108 h - 334) Q_0(x) + (-18 h + 171) Q_1(x) + (-72) Q_2(x) \\
              & + (h - 9) Q_3(x), \\
   Q_2(h x) = & ~ (3 h - 27) Q_0(x) + 8 Q_1(x) + 9 Q_2(x) + (-24) Q_3(x), \\
   Q_3(h x) = & ~ (h^2 - 18 h + 57) Q_0(x) + (3 h - 27) Q_1(x) + 8 Q_2(x) + 9 Q_3(x), \\
   Q_0(c x) = & ~ (c^3 - 12 c + 12 c^{-1}) Q_0(x) + (3 c - 12 c^{-1}) Q_1(x) + (12 c^{-1}) Q_2(x) \\
              & + (-12 c^{-1}) Q_3(x), \\
   Q_1(c x) = & ~ (-6 c + 20 c^{-1}) Q_0(x) + (-20 c^{-1}) Q_1(x) + (- c + 20 c^{-1}) Q_2(x) \\
              & + (4 c - 20 c^{-1}) Q_3(x), \\
   Q_2(c x) = & ~ (4 c^{-1}) Q_0(x) + (-4 c^{-1}) Q_1(x) + (4 c^{-1}) Q_2(x) + (- c - 4 c^{-1}) Q_3(x), \\
   Q_3(c x) = & ~ (c - 4 c^{-1}) Q_0(x) + (4 c^{-1}) Q_1(x) + (-4 c^{-1}) Q_2(x) + (4 c^{-1}) Q_3(x), \\
   Q_k(i x) = & ~ (-i) Q_k(x) \text{~for all~} k; \\
  \end{split}
  \end{equation*}

  \item \label{Q(iv)} {\em Adem relations }
  \begin{equation*}
  \begin{split}
   Q_1Q_0(x) = & ~ (-6) Q_0Q_1(x) + 3 Q_2Q_1(x) + (6 h - 54) Q_0Q_2(x) + 18 Q_1Q_2(x) \\
               & + (-9) Q_3Q_2(x) + (-6 h^2 + 108 h - 369) Q_0Q_3(x) \\
               & + (-18 h + 162) Q_1Q_3(x) + (-54) Q_2Q_3(x), \\
   Q_2Q_0(x) = & ~ 3 Q_3Q_1(x) + (-3) Q_0Q_2(x) + (3 h - 27) Q_0Q_3(x) + 9 Q_1Q_3(x), \qquad \qquad \\
   Q_3Q_0(x) = & ~ Q_0Q_1(x) + (-h + 9) Q_0Q_2(x) + (-3) Q_1Q_2(x) \\
               & + (h^2 - 18 h + 63) Q_0Q_3(x) + (3 h - 27) Q_1Q_3(x) + 9 Q_2Q_3(x); 
  \end{split}
  \end{equation*}

  \item \label{Q(v)} {\em Cartan formulas }
  \begin{equation*}
  \begin{split}
   Q_0(xy) = & ~ Q_0(x) Q_0(y) + 3 \big( Q_3(x) Q_1(y) + Q_2(x) Q_2(y) + Q_1(x) Q_3(y) \big) \\
             & + 18 Q_3(x) Q_3(y), \\
   Q_1(xy) = & ~ \big( Q_1(x) Q_0(y) + Q_0(x) Q_1(y) \big) \\
             & + (-h + 9) \big( Q_3(x) Q_1(y) + Q_2(x) Q_2(y) + Q_1(x) Q_3(y) \big) \\
             & + 3 \big( Q_3(x) Q_2(y) + Q_2(x) Q_3(y) \big) + (-6 h + 54) Q_3(x) Q_3(y), \qquad \qquad \qquad
  \end{split}
  \end{equation*}
  \begin{equation*}
  \begin{split}
   Q_2(xy) = & ~ \big( Q_2(x) Q_0(y) + Q_1(x) Q_1(y) + Q_0(x) Q_2(y) \big) \\
             & + 6 \big( Q_3(x) Q_1(y) + Q_2(x) Q_2(y) + Q_1(x) Q_3(y) \big) \\
             & + (-h + 9) \big( Q_3(x) Q_2(y) + Q_2(x) Q_3(y) \big) + 39 Q_3(x) Q_3(y), \\
   Q_3(xy) = & ~ \big( Q_3(x) Q_0(y) + Q_2(x) Q_1(y) + Q_1(x) Q_2(y) + Q_0(x) Q_3(y) \big) \qquad \qquad \qquad \\
             & + 6 \big( Q_3(x) Q_2(y) + Q_2(x) Q_3(y) \big) + (-h + 9) Q_3(x) Q_3(y); 
  \end{split}
  \end{equation*}

  \item \label{Q(vi)} {\em The Frobenius congruence }
  \begin{equation*}
   Q_0(x) \equiv x^3 \md 3.  \qquad \qquad \qquad \qquad \qquad \qquad \qquad \qquad \qquad \qquad \qquad \qquad
  \end{equation*}
 \end{enumerate}
\end{prop}
\begin{proof}
 The relations in \eqref{Q(i)}, \eqref{Q(ii)}, \eqref{Q(iii)}, and 
 \eqref{Q(v)} follow computationally from the formulas in Corollary 
 \ref{cor:psi3} together with the fact that $\p$ is a ring homomorphism.  

 For \eqref{Q(iv)}, there is a canonical isomorphism $C/C[3] \cong C$ of 
 elliptic curves.  Given the correspondence between deformations of 
 Frobenius and power operations in \cite[Theorem B]{cong}, the 
 commutativity of \eqref{frob^2} then implies that the composite 
 \eqref{psi3^2} lands in $A_0$.  In terms of formulas, we have 
 \begin{equation*}
 \begin{split}
  \p \big( \p(x) \big) = & ~ \p \big( Q_0(x) + Q_1(x) \A + Q_2(x) \A^2 + Q_3(x) \A^3 \big) \\
                       = & ~ \sum_{k = 0}^3 \p \big( Q_k(x) \big) \big( \p(\A) \big)^k \\
                       = & ~ \sum_{k = 0}^3 \sum_{j = 0}^3 Q_jQ_k(x) \A^j (-\A^3 + 6 \A - h + 9)^k \\
                  \equiv & ~ \Psi_0(x) + \Psi_1(x) \A + \Psi_2(x) \A^2 + \Psi_3(x) \A^3 \md \big( w(\A) \big) 
 \end{split}
 \end{equation*}
 where each $\Psi_i$ is an $E_0$-linear combination of the $Q_jQ_k$'s.  
 The vanishing of $\Psi_1(x)$, $\Psi_2(x)$, and $\Psi_3(x)$ gives the 
 three relations in \eqref{Q(iv)}.  

 For \eqref{Q(vi)}, by \cite[Propositions 3.25 and 10.5]{cong} we have 
 \[
  \p(x) \equiv x^3 \md 3.  
 \]
 In view of \eqref{Amod3}, the congruence in \eqref{Q(vi)} then follows 
 from \eqref{Q_k}.  
\end{proof}

\begin{ex}
\label{ex}
 We have $E^0 S^2 \cong \BZ_9 \llbracket h \rrbracket [u] / (u^2)$.  By 
 definition of $\K$ in \eqref{KL}, the $Q_k$'s act canonically on 
 $u \in E^0 S^2$: 
 \[
  Q_k(u) = \left\{
  \begin{array}{ll}
    u,  & \quad {\rm if}~k = 1, \\
    0,  & \quad {\rm if}~k \neq 1.  \\
  \end{array}
  \right.
 \]
 We then get the values of the $Q_k$'s on elements in $E^0 S^2$ from 
 \q{i}-\eqref{Q(iii)}.  
\end{ex}

\subsection{The \DL algebra}

\begin{defn}
\label{def:go}
 \mbox{}
 \begin{enumerate}[(i)]
  \item \label{go(i)} Let $i$ be an element generating $\BZ_9$ over 
  $\BZ_3$ with $i^2 = -1$.  Define $\g$ to be the associative ring 
  generated over $\BZ_9 \llbracket h \rrbracket$ by elements $q_0$, 
  $q_1$, $q_2$, and $q_3$ subject to the following relations: the 
  $q_k$'s commute with elements in 
  $\BZ_3 \subset \BZ_9 \llbracket h \rrbracket$, and satisfy {\em 
  commutation relations} 
  \begin{equation*}
  \begin{split}
   q_0 h = & ~ (h^3 - 27 h^2 + 201 h - 342) q_0 + (3 h^2 - 54 h + 171) q_1 + (9 h - 81) q_2 \\
           & + 24 q_3, \\
   q_1 h = & ~ (-6 h^2 + 108 h - 334) q_0 + (-18 h + 171) q_1 + (-72) q_2 + (h - 9) q_3, \\
   q_2 h = & ~ (3 h - 27) q_0 + 8 q_1 + 9 q_2 + (-24) q_3, \\
   q_3 h = & ~ (h^2 - 18 h + 57) q_0 + (3 h - 27) q_1 + 8 q_2 + 9 q_3, \\
   q_k i ~ = & ~ (-i) q_k \text{~for all~} k, 
  \end{split}
  \end{equation*}
  and {\em Adem relations} 
  \begin{equation*}
  \begin{split}
   q_1q_0 = & ~ (-6) q_0q_1 + 3 q_2q_1 + (6 h - 54) q_0q_2 + 18 q_1q_2 + (-9) q_3q_2 \\
            & + (-6 h^2 + 108 h - 369) q_0q_3 + (-18 h + 162) q_1q_3 + (-54) q_2q_3, \quad~~ \\
   q_2q_0 = & ~ 3 q_3q_1 + (-3) q_0q_2 + (3 h - 27) q_0q_3 + 9 q_1q_3, \\
   q_3q_0 = & ~ q_0q_1 + (-h + 9) q_0q_2 + (-3) q_1q_2 + (h^2 - 18 h + 63) q_0q_3 \\
            & + (3 h - 27) q_1q_3 + 9 q_2q_3.  
  \end{split}
  \end{equation*}

  \item \label{go(ii)} Write $\omega \coloneqq \pi_2 E$ which is the 
  kernel of $E^0 S^2 \to E^0$ with 
  $E^0 S^2 \cong \BZ_9 \llbracket h \rrbracket [u] / (u^2)$.  Define 
  $\omega$ as a $\g$-module in the sense of \cite[2.2]{h2p2} with one 
  generator $u$ by 
  \[
   q_k \cdot u = \left\{
   \begin{array}{ll}
     u,  & \quad {\rm if}~k = 1, \\
     0,  & \quad {\rm if}~k \neq 1.  \\
   \end{array}
   \right.
  \]
 \end{enumerate}
\end{defn}

\begin{rmk}
\label{rmk:rank}
 In \go{i}, an element $r \in \BZ_9 \llbracket h \rrbracket \cong E_0$ 
 corresponds to the multiplication-by-$r$ operation (see 
 \cite[Proposition 6.4]{cong}), and each $q_k$ corresponds to the 
 individual power operation $Q_k$ (also cf.~\go{ii} and Example 
 \ref{ex}).  Under this correspondence, the relations in 
 \q{ii}-\eqref{Q(v)}  describe explicitly the structure of $\g$ as that 
 of a {\em graded twisted bialgebra over $E_0$} in the sense of 
 \cite[Section 5]{cong}.  The grading of $\g$ comes from the number of 
 the $q_k$'s in a monomial: for example, commutation relations are in 
 degree 1, and Adem relations are in degree 2.  Under these relations, 
 $\g$ has an {\em admissible basis}: it is free as a left $E_0$-module 
 on the elements of the form 
 \[
  q_0^m q_{k_1} \cdots q_{k_n} 
 \]
 where $m, n \geq 0$ ($n = 0$ gives $q_0^m$), and $k_i = 1$, 2, or 3.  
 If we write $\g[d]$ for the degree-$d$ part of $\g$, then $\g[d]$ is of 
 rank $1 + 3 + \cdots + 3^d$.  
\end{rmk}

We now identify $\g$ with the \DL algebra of power operations on 
$K(2)$-local commutative $E$-algebras.  

\begin{thm}
\label{thm:gamma}
 Let $A$ be a $K(2)$-local commutative $E$-algebra.  Let $\g$ be the 
 graded twisted bialgebra over $E_0$ as defined in \go{i}, and $\omega$ 
 be the $\g$-module in \go{ii}.  Then $A_*$ is an {\em $\omega$-twisted 
 $\BZ/2$-graded amplified $\g$-ring} in the sense of 
 \cite[Section 2]{cong} and \cite[2.5 and 2.6]{h2p2}.  In particular, 
 \[
  \pi_* L_{K(2)} \BP_E (\Sigma^d E) \cong \big( F_d \big)_{(3,h)}^\wedge, 
 \]
 where $F_d$ is the free $\omega$-twisted $\BZ/2$-graded amplified 
 $\g$-ring with one generator in degree $d$.  
\end{thm}
Formulas for $\g$ aside, this result is due to Rezk \cite{cong, h2p2}.  
\begin{proof}
 Let $\G$ be the graded twisted bialgebra of power operations on $E_0$ 
 in \cite[Section 6]{cong}.  We need only identify $\G$ with $\g$.  

 There is a direct sum decomposition $\G = \bigoplus_{d \geq 0} \G[d]$ 
 where the summands come from the completed $E$-homology of 
 $B\Sigma_{3^d}$ (see \cite[6.2]{cong}).  As in Remark \ref{rmk:rank}, 
 we have a degree-preserving ring homomorphism 
 \[
  \phi \co \g \to \G, \qquad q_k \mapsto Q_k 
 \]
 which is an isomorphism in degrees 0 and 1.  We need to show that 
 $\phi$ is both surjective and injective in all degrees.  

 For the surjectivity of $\phi$, we use a transfer argument.  We have 
 \[
  \nu_3(|\Sigma_3^{\wr d}|) = \nu_3(|\Sigma_{3^d}|) = (3^d - 1) / 2 
 \]
 where $\nu_3(-)$ is the 3-adic valuation, and $(-)^{\wr d}$ is the 
 $d$-fold wreath product.  Thus following the proof of 
 \cite[Proposition 3.17]{cong}, we see that $\G$ is generated in degree 
 1, and hence $\phi$ is surjective.  

 By Remark \ref{rmk:rank} and (the $E_0$-linear dual of) 
 \cite[Theorem 1.1]{Str98}, $\g[d]$ and $\G[d]$ are of the same rank 
 $1 + 3 + \cdots + 3^d$ as free modules over $E_0$.  Hence $\phi$ is 
 also injective.  
\end{proof}

\section{$K(1)$-local power operations}
\label{sec:K(1)}

Let $F \coloneqq L_{K(1)} E$ be the $K(1)$-localization of $E$.  The 
following diagram describes the relationship between $K(1)$-local power 
operations on $F^0$ (cf.~\cite[Section 3]{hopkins} and 
\cite[Section IX.3]{H_infty}) and the power operation on $E^0$ in 
Corollary \ref{cor:psi3}: 
\begin{center}
\begin{tikzpicture}
        \node (LT) at (0, 2) {$E^0$}; 
        \node (RT) at (3, 2) {$E^0 B\Sigma_3 / I$}; 
        \node (LB) at (0, 0) {$F^0$}; 
        \node (MB) at (3, 0) {$F^0 B\Sigma_3 / J$}; 
        \node (RB) at (4.3, 0) {$\cong F^0.  $}; 
        \draw [->] (LT) -- node [above] {$\scriptstyle \p$} (RT); 
        \draw [->] (LT) -- (LB); 
        \draw [->] (RT) -- (MB); 
        \draw [->] (LB) -- node [above] {$\scriptstyle \psi_F^3$} (MB); 
\end{tikzpicture}
\end{center}
Here $\psi_F^3$ is the $K(1)$-local power operation induced by $\p$, and 
$J \cong F^0 \otimes_{E^0} I$ is the transfer ideal 
(cf.~\eqref{transfer}).  Recall from \isog{i}, \eqref{S_3}, and 
Corollary \ref{cor:psi3} that $\p$ arises from the universal degree-3 
isogeny which is parametrized by the ring $\s_3$ with 
\[
 \big( S_3 \big)_{(3,h)}^\wedge \cong E^0 B\Sigma_3 / I.  
\]
The vertical maps are induced by the $K(1)$-localization $E \to F$.  In 
terms of homotopy groups, this is obtained by inverting the generator 
$h$ and completing at the prime 3 (see \cite[Corollary 1.5.5]{hovey}): 
\[
 E_* = \BZ_9 \llbracket h \rrbracket [u^{\pm1}] \qquad \ad \qquad F_* = \BZ_9 \llbracket h \rrbracket [h^{-1}]_3^\wedge [u^{\pm1}] 
\]
with 
\[
 F_0 = \BZ_9 (\!(h)\!)_3^\wedge = \left.\left\{\sum_{n = -\infty}^{\infty} k_n h^n~\right|~k_n \in \BZ_9, \lim_{n \to -\infty} k_n = 0\right\}.  
\]
The formal group $\HC$ over $E^0$ has a unique order-3 subgroup after 
being pulled back to $F^0$ (cf.~Remark \ref{rmk:dmod3}), and the map 
\[
 E^0 B\Sigma_3 / I \to F^0 B\Sigma_3 / J \cong F^0 
\]
classifies this subgroup.  Along the base change 
\[
 E^0 B\Sigma_3 / I \to F^0 \otimes_{E^0} (E^0 B\Sigma_3 / I) \cong (F^0 \otimes_{E^0} E^0 B\Sigma_3) / J \cong F^0 B\Sigma_3 / J, 
\]
the special fiber of the 3-divisible group of $\HC$ which consists 
solely of a formal component may split into formal and \'etale 
components.  We want to take the formal component so as to keep track of 
the unique order-3 subgroup of the formal group over $F^0$.  This 
subgroup gives rise to the $K(1)$-local power operation $\psi_F^3$.  

Recall from \eqref{S_3} that $S_3 = S[\A] \big/ \big( w(\A) \big)$.  
Since 
\[
 w(\A) = \A^4 - 6 \A^2 + (h - 9) \A - 3 \equiv \A (\A^3 + h) \md 3, 
\]
the equation $w(\A) = 0$ has a unique root $\A = 0$ in $\BF_9 (\!(h)\!)$ 
(cf.~\eqref{Amod3}).  By Hensel's lemma this unique root lifts to a root 
in $\BZ_9 (\!(h)\!)_3^\wedge$; it corresponds to the unique order-3 
subgroup of $\HC$ over $F^0$.  Plugging this specific value of $\A$ into 
the formulas for $\p$ in Corollary \ref{cor:psi3}, we then get an 
endomorphism of the ring $F^0$.  This endomorphism is the $K(1)$-local 
power operation $\psi_F^3$.  

Explicitly, with $h$ invertible in $F^0$, we solve for $\A$ from 
$w(\A) = 0$ by first writing 
\[
 \A = (3 + 6 \A^2 - \A^4) / (h - 9) = (3 + 6 \A^2 - \A^4) \sum_{n = 1}^\infty 9^{n-1} h^{-n} 
\]
and then substituting this equation into itself recursively.  We plug 
the power series expansion for $\A$ into $\p(h)$ and get 
\[
 \psi_F^3(h) = h^3 - 27 h^2 + 183 h - 180 + 186 h^{-1} + 1674 h^{-2} + (\text{lower-order terms}).  ~~~
\]
Similarly, writing $h$ as $c^2 + 1$ in $w(\A) = 0$, we solve for $\A$ in 
terms of $c$ and get 
\[
 \psi_F^3(c) = c^3 - 12 c - 6 c^{-1} - 84 c^{-3} - 933 c^{-5} - 10956 c^{-7} + (\text{lower-order terms}).  
\]

\appendix
\section*{Appendices}

Here we list long formulas whose appearance in the main body might 
affect readability.  The calculations involve power series expansions 
and manipulations of long polynomials with large coefficients (division, 
factorization, and finding greatest common divisors).  They are done 
using the software {\em Wolfram Mathematica 8}.  The commands 
\texttt{Reduce} and \texttt{Solve} are used to extract relations out of 
given identities.

\section{Formulas in the proof of Proposition 4}
\label{apx:tors}

In the proof of Proposition \ref{prop:tors}, we have 
\begin{equation*}
\begin{split}
 \Tf(u) = & -\frac{u^4}{a^2 b} \big( b^4 u^8 + 3 a b^3 u^7 + 3 a^2 b^2 u^6 + (a^3 b + 7 a b^2) u^5 + (6 a^2 b - 6 b^2) u^4 \qquad~~ \\
          & + 9 a b u^3 + (-a^2 + 8 b) u^2 - 3 a u - 3 \big), \\
 Q_1(v) = & ~ a b^2 v^2 + (b^2 d^2 + 2 a b d - b) v + \frac{b^2 d^4}{a} + 2 b d^3 + a d^2 - \frac{2 b d^2}{a} - d + \frac{1}{a}, 
\end{split}
\end{equation*}
\begin{equation*}
\begin{split}
 R_1(v) = & ~ (\frac{b^3 d^6}{a} + 2 b^2 d^5 + a b d^4 - \frac{3 b^2 d^4}{a} + 2 b d^3 + \frac{3 b d^2}{a} - \frac{1}{a}) v + \frac{b^2 d^7}{a} + 2 b d^6 \\
          & + a d^5 - \frac{2 b d^5}{a} + 2 d^4 + \frac{d^3}{a}, \\
 Q_2(v) = & ~ \frac{a}{(b^3 d^6 + 2 a b^2 d^5 + a^2 b d^4 - 3 b^2 d^4 + 2 a b d^3 + 3 b d^2 - 1)^2} \big( (a b^4 d^6 + 2 a^2 b^3 d^5 \\
          & + a^3 b^2 d^4 - 3 a b^3 d^4 + 2 a^2 b^2 d^3 + 3 a b^2 d^2 - a b) v - b^4 d^8 - 2 a b^3 d^7 - a^2 b^2 d^6 \\
          & + 4 b^3 d^6 - a b^2 d^5 + a^2 b d^4 - 6 b^2 d^4 + 4 a b d^3 + 4 b d^2 - a d - 1 \big), \\
 R_2 = & - \frac{a d^4}{(b^3 d^6 + 2 a b^2 d^5 + a^2 b d^4 - 3 b^2 d^4 + 2 a b d^3 + 3 b d^2 - 1)^2} (b^4 d^8 + 3 a b^3 d^7 \\
       & + 3 a^2 b^2 d^6 + a^3 b d^5 + 7 a b^2 d^5 + 6 a^2 b d^4 - 6 b^2 d^4 + 9 a b d^3 - a^2 d^2 + 8 b d^2 \\
       & - 3 a d - 3), \\
 K(u) = & ~ \frac{b^3 u^6}{a} + 2 b^2 u^5 + (a b - \frac{3 b^2}{a}) u^4 + 2 b u^3 + \frac{3 b u^2}{a} - \frac{1}{a}, \\
 L(u) = & ~ \frac{b^2 u^7}{a} + 2 b u^6 + (a - \frac{2 b}{a}) u^5 + 2 u^4 + \frac{u^3}{a}, \\
 M(u) = & ~ \frac{b}{a^2 (a^2 - 16 b)^2} \big( (10 a^3 b^3 - 112 a b^4) u^5 + (19 a^4 b^2 - 217 a^2 b^3 - 16 b^4) u^4 \\
        & + (8 a^5 b - 126 a^3 b^2 + 304 a b^3) u^3 + (-a^6 + 34 a^4 b -266 a^2 b^2 +32 b^3) u^2 \\
        & + (28 a^3 b - 384 a b^2) u - 4 a^4 + 51 a^2 b - 16 b^2 \big), \\
 N(u) = & -\frac{1}{a (a^2 - 16 b)^2} \big( (10 a^3 b^5 - 112 a b^6) u^7 + (29 a^4 b^4 - 329 a^2 b^5 - 16 b^6) u^6 \\
        & + (27 a^5 b^3 - 313 a^3 b^4 - 48 a b^5 ) u^5 + (7 a^6 b^2 - 15 a^4 b^3 - 837 a^2 b^4 - 16 b^5) u^4 \\
        & + (-a^7 b + 66 a^5 b^2 - 714 a^3 b^3 + 528 a b^4) u^3 + (-4 a^6 b + 137 a^4 b^2 \\
        & - 1147 a^2 b^3 + 80 b^4) u^2 + (-12 a^5 b + 237 a^3 b^2 - 1200 a b^3) u + a^6 - 44 a^4 b \\
        & + 409 a^2 b^2 - 48 b^3 \big).  
\end{split}
\end{equation*}

\section{Formulas in the proof of Proposition 6}
\label{apx:isog}

In the proof of Proposition \ref{prop:isog}, the power series expansion 
of $v$ in terms of $u$ up to $u^{12}$ is 
\begin{equation*}
\begin{split}
 v = & ~ u^3 - a u^4 + (a^2 + b) u^5 + (-a^3 - 3 a b) u^6 + (a^4 + 6 a^2 b + b^2) u^7 + (-a^5 - 10 a^3 b \\
     & - 6 a b^2) u^8 + (a^6 + 15 a^4 b + 20 a^2 b^2 + b^3) u^9 + (-a^7 - 21 a^5 b - 50 a^3 b^2 \\
     & - 10 a b^3) u^{10} + (a^8 + 28 a^6 b + 105 a^4 b^2 + 50 a^2 b^3 + b^4) u^{11} + (-a^9 - 36 a^7 b \\
     & - 196 a^5 b^2 - 175 a^3 b^3 - 15 a b^4) u^{12}.  
\end{split}
\end{equation*}

The group law on $C$ satisfies: 
\begin{itemize}
 \item given $P(u,v)$, the coordinates of $-P$ are 
 \[
  \left( -\frac{v}{u (u + b v)},-\frac{v^2}{u^2 (u + b v)} \right); 
 \]

 \item given $P_1(u_1,v_1)$ and $P_2(u_2,v_2)$, the coordinates of 
 $-(P_1 + P_2)$ are 
 \[
  u_3 \coloneqq a k - \frac{b m}{1 + b k} - u_1 - u_2 \qquad \ad \qquad v_3 \coloneqq k u_3 + m 
 \]
 where 
 \[
  k = \frac{v_1 - v_2}{u_1 - u_2} \qquad \ad \qquad m = \frac{u_1 v_2 - u_2 v_1}{u_1 - u_2}.  
 \]
\end{itemize}
Given $P(u,v)$ and $Q(d,e)$, with the above notations and formulas, 
\begin{itemize}
 \item set 
 \[
  (u_1,v_1) = \left( -\frac{v}{u (u + b v)},-\frac{v^2}{u^2 (u + b v)} \right) \qquad \ad \qquad (u_2,v_2) = (d,e) 
 \]
 so that 
 \[
  P - Q = (u_3,v_3); 
 \]

 \item set 
 \[
  (u_1,v_1) = (u,v) \qquad \ad \qquad (u_2,v_2) = (d,e) 
 \]
 so that 
 \[
  P + Q = \left( -\frac{v_3}{u_3 (u_3 + b v_3)},-\frac{v_3^2}{u_3^2 (u_3 + b v_3)} \right).  
 \]
\end{itemize}
Plugging the coordinates of $P - Q$ and $P + Q$ into \eqref{u'v'}, in 
view of \eqref{f}, we then have in \eqref{KL} 
\begin{equation*}
\begin{split}
      \K = & -\frac{1}{a^2 - 16 b} \big( a b^3 d^7 + (3 a^2 b^2 - 2 b^3) d^6 + (3 a^3 b - 6 a b^2) d^5 + (a^4 + a^2 b + 2 b^2) d^4 \\
           & + (4 a^3 - 15 a b) d^3 + (a^2 + 2 b) d^2 - 12 a d - 18 \big), \\
 \lambda = & -\frac{1}{a^2 b^2 (a^2 - 16 b)} \big( (a^3  b^3 - 11 a b^4) d^7 + (3 a^4 b^2 - 33 a^2 b^3 - 4 b^4) d^6 + (3 a^5 b \\
           & - 33 a^3 b^2 - 15 a b^3) d^5 + (a^6 - 4 a^4 b - 96 a^2 b^2 - 4 b^3) d^4 + (6 a^5 - 80 a^3 b \\
           & + 31 a b^2) d^3 + (10 a^4 - 153 a^2 b + 20 b^2) d^2 + (3 a^3 - 117 a b) d - 6 a^2 - 12 b \big).  
\end{split}
\end{equation*}
More extended power series expansions in $u$ for $u'$ (up to $u^6$) and 
$v'$ (up to $u^9$) are needed in \eqref{KL} to determine the 
coefficients in the equation of $C'$: 
\begin{equation*}
\begin{split}
 u' = & -\frac{1}{a^2 - 16 b} \big( (a b^3 d^7 + 3 a^2 b^2 d^6 - 2 b^3 d^6 + 3 a^3 b d^5 - 6 a b^2 d^5 + a^4 d^4 + a^2 b d^4 \quad~~~ \\
      & + 2 b^2 d^4 + 4 a^3 d^3 - 15 a b d^3 + a^2 d^2 + 2 b d^2 - 12 a d - 18) u + (-a^2 b^3 d^7 
\end{split}
\end{equation*}
\begin{equation*}
\begin{split}
      & + 12 b^4 d^7 - 3 a^3 b^2 d^6 + 36 a b^3 d^6 - 3 a^4 b d^5 + 36 a^2 b^2 d^5 + 4 b^3 d^5 - a^5 d^4 \\
      & + 5 a^3 b d^4 + 94 a b^2 d^4 - 6 a^4 d^3 + 85 a^2 b d^3 - 76 b^2 d^3 - 9 a^3 d^2 + 136 a b d^2 + 60 b d \\
      & + 6 a) u^2 + (a^3 b^3 d^7 - 17 a b^4 d^7 + 3 a^4 b^2 d^6 - 50 a^2 b^3 d^6 - 8 b^4 d^6 + 3 a^5 b d^5 \\
      & - 48 a^3 b^2 d^5 - 27 a b^3 d^5 + a^6 d^4 - 7 a^4 b d^4 - 150 a^2 b^2 d^4 - 16 b^3 d^4 + 7 a^5 d^3 \\
      & - 113 a^3 b d^3 + 9 a b^2 d^3 + 16 a^4 d^2 - 258 a^2 b d^2 + 56 b^2 d^2 + 15 a^3 d - 237 a b d \\
      & + 2 a^2 - 32 b) u^3 + (-a^4 b^3 d^7 + 16 a^2 b^4 d^7 + 12 b^5 d^7 - 3 a^5 b^2 d^6 + 46 a^3 b^3 d^6 \\
      & + 64 a b^4 d^6 - 3 a^6 b d^5 + 42 a^4 b^2 d^5 + 121 a^2 b^3 d^5 + 4 b^4 d^5 - a^7 d^4 + 3 a^5 b d^4 \\
      & + 209 a^3 b^2 d^4 + 122 a b^3 d^4 - 8 a^6 d^3 + 114 a^4 b d^3 + 248 a^2 b^2 d^3 - 76 b^3 d^3 \\
      & - 24 a^5 d^2 + 384 a^3 b d^2 - 4 a b^2 d^2 - 33 a^4 d + 519 a^2 b d + 60 b^2 d - 18 a^3 \\
      & + 282 a b) u^4 + (a^5 b^3 d^7 - 9 a^3 b^4 d^7 - 117 a b^5 d^7 + 3 a^6 b^2 d^6 - 24 a^4 b^3 d^6 \\
      & - 396 a^2 b^4 d^6 - 24 b^5 d^6 + 3 a^7 b d^5 - 18 a^5 b^2 d^5 - 484 a^3 b^3 d^5 - 111 a b^4 d^5 + a^8 d^4 \\
      & + 7 a^6 b d^4 - 307 a^4 b^2 d^4 - 1038 a^2 b^3 d^4 + 9 a^7 d^3 - 73 a^5 b d^3 - 1181 a^3 b^2 d^3 \\
      & + 573 a b^3 d^3 + 33 a^6 d^2 - 451 a^4 b d^2 - 1236 a^2 b^2 d^2 + 72 b^3 d^2 + 54 a^5 d \\
      & - 807 a^3 b d - 873 a b^2 d + 36 a^4 - 570 a^2 b - 48 b^2) u^5 + (-a^6 b^3 d^7 - 5 a^4 b^4 d^7 \\
      & + 337 a^2 b^5 d^7 + 12 b^6 d^7 - 3 a^7 b^2 d^6 - 19 a^5 b^3 d^6 + 1064 a^3 b^4 d^6 + 204 a b^5 d^6 \\
      & - 3 a^8 b d^5 - 27 a^6 b^2 d^5 + 1164 a^4 b^3 d^5 + 638 a^2 b^4 d^5 + 4 b^5 d^5 - a^9 d^4 - 24 a^7 b d^4 \\
      & + 441 a^5 b^2 d^4 + 3195 a^3 b^3 d^4 + 182 a b^4 d^4 - 10 a^8 d^3 - 22 a^6 b d^3 + 2956 a^4 b^2 d^3 \\
      & - 645 a^2 b^3 d^3 - 76 b^4 d^3 - 43 a^7 d^2 + 403 a^5 b d^2 + 4594 a^3 b^2 d^2 - 544 a b^3 d^2 \\
      & - 78 a^6 d + 996 a^4 b d + 4014 a^2 b^2 d + 60 b^3 d - 57 a^5 + 852 a^3 b + 942 a b^2) u^6 \big), \\
 v' = & -\frac{1}{a^2 b^2 (a^2 - 16 b)} \big( (a^3 b^3 d^7 - 11 a b^4 d^7 + 3 a^4 b^2 d^6 - 33 a^2 b^3 d^6 - 4 b^4 d^6 \\
      & + 3 a^5 b d^5 - 33 a^3 b^2 d^5 - 15 a b^3 d^5 + a^6 d^4 - 4 a^4 b d^4 - 96 a^2 b^2 d^4 - 4 b^3 d^4 \\
      & + 6 a^5 d^3 - 80 a^3 b d^3 + 31 a b^2 d^3 + 10 a^4 d^2 - 153 a^2 b d^2 + 20 b^2 d^2 + 3 a^3 d \\
      & - 117 a b d - 6 a^2 - 12 b) u^3 + (-2 a^4 b^3 d^7 + 28 a^2 b^4 d^7 - 6 a^5 b^2 d^6 + 82 a^3 b^3 d^6 \\
      & + 28 a b^4 d^6 - 6 a^6 b d^5 + 78 a^4 b^2 d^5 + 90 a^2 b^3 d^5 - 2 a^7 d^4 + 8 a^5 b d^4 + 294 a^3 b^2 d^4 \\
      & + 20 a b^3 d^4 - 14 a^6 d^3 + 202 a^4 b d^3 + 72 a^2 b^2 d^3 - 32 a^5 d^2 + 510 a^3 b d^2 \\
      & - 124 a b^2 d^2 - 30 a^4 d + 546 a^2 b d - 6 a^3 + 204 a b) u^4 + (3 a^5 b^3 d^7 - 38 a^3 b^4 d^7 \\
      & - 107 a b^5 d^7 + 9 a^6 b^2 d^6 - 108 a^4 b^3 d^6 - 409 a^2 b^4 d^6 - 4 b^5 d^6 + 9 a^7 b d^5 \\
      & - 96 a^5 b^2 d^5 - 590 a^3 b^3 d^5 - 47 a b^4 d^5 + 3 a^8 d^4 + a^6 b d^4 - 646 a^4 b^2 d^4 \\
      & - 912 a^2 b^3 d^4 - 4 b^4 d^4 + 24 a^7 d^3 - 292 a^5 b d^3 - 1249 a^3 b^2 d^3 + 639 a b^3 d^3 
\end{split}
\end{equation*}
\begin{equation*}
\begin{split}
~~~~~~~& + 70 a^6 d^2 - 1057 a^4 b d^2 - 849 a^2 b^2 d^2 + 20 b^3 d^2 + 93 a^5 d - 1512 a^3 b d \\
      & - 597 a b^2 d + 48 a^4 - 870 a^2 b - 12 b^2) u^5 + (-4 a^6 b^3 d^7 + 24 a^4 b^4 d^7 + 583 a^2 b^5 d^7 \\
      & - 12 a^7 b^2 d^6 + 60 a^5 b^3 d^6 + 1923 a^3 b^4 d^6 + 156 a b^5 d^6 - 12 a^8 b d^5 + 36 a^6 b^2 d^5 \\
      & + 2268 a^4 b^3 d^5 + 639 a^2 b^4 d^5 - 4 a^9 d^4 - 40 a^7 b d^4 + 1256 a^5 b^2 d^4 + 5128 a^3 b^3 d^4 \\
      & + 140 a b^4 d^4 - 36 a^8 d^3 + 229 a^6 b d^3 + 5409 a^4 b^2 d^3 - 2227 a^2 b^3 d^3 - 127 a^7 d^2 \\
      & + 1597 a^5 b d^2 + 6835 a^3 b^2 d^2 - 748 a b^3 d^2 - 201 a^6 d + 2952 a^4 b d + 5277 a^2 b^2 d \\
      & - 129 a^5 + 2130 a^3 b + 708 a b^2) u^6 + (5 a^7 b^3 d^7 + 35 a^5 b^4 d^7 - 1754 a^3 b^5 d^7 \\
      & - 275 a b^6 d^7 + 15 a^8 b^2 d^6 + 125 a^6 b^3 d^6 - 5511 a^4 b^4 d^6 - 1833 a^2 b^5 d^6 - 4 b^6 d^6 \\
      & + 15 a^9 b d^5 + 165 a^7 b^2 d^5 - 5988 a^5 b^3 d^5 - 4312 a^3 b^4 d^5 - 103 a b^5 d^5 + 5 a^{10} d^4 \\
      & + 130 a^8 b d^4 - 2183 a^6 b^2 d^4 - 17022 a^4 b^3 d^4 - 2940 a^2 b^4 d^4 - 4 b^5 d^4 + 50 a^9 d^3 \\
      & + 159 a^7 b d^3 - 15035 a^5 b^2 d^3 + 179 a^3 b^3 d^3 + 1703 a b^4 d^3 + 206 a^8 d^2 \\
      & - 1708 a^6 b d^2 - 25304 a^4 b^2 d^2 + 1431 a^2 b^3 d^2 + 20 b^4 d^2 + 363 a^7 d - 4398 a^5 b d \\
      & - 23694 a^3 b^2 d - 1437 a b^3 d + 258 a^6 - 3816 a^4 b - 7026 a^2 b^2 - 12 b^3) u^7 \\
      & + (-6 a^8 b^3 d^7 - 164 a^6 b^4 d^7 + 3864 a^4 b^5 d^7 + 3365 a^2 b^6 d^7 - 18 a^9 b^2 d^6 \\
      & - 522 a^7 b^3 d^6 + 11837 a^5 b^4 d^6 + 13701 a^3 b^5 d^6 + 448 a b^6 d^6 - 18 a^{10} b d^5 \\
      & - 582 a^8 b^2 d^5 + 12275 a^6 b^3 d^5 + 21828 a^4 b^4 d^5 + 2395 a^2 b^5 d^5 - 6 a^{11} d^4 \\
      & - 296 a^9 b d^4 + 3283 a^7 b^2 d^4 + 43960 a^5 b^3 d^4 + 30290 a^3 b^4 d^4 + 424 a b^5 d^4 \\
      & - 66 a^{10} d^3 - 1099 a^8 b d^3 + 32246 a^6 b^2 d^3 + 30529 a^4 b^3 d^3 - 17045 a^2 b^4 d^3 \\
      & - 310 a^9 d^2 + 679 a^7 b d^2 + 66726 a^5 b^2 d^2 + 24833 a^3 b^3 d^2 - 2192 a b^4 d^2 - 588 a^8 d \\
      & + 4809 a^6 b d + 73578 a^4 b^2 d + 23685 a^2 b^3 d - 444 a^7 + 5316 a^5 b + 30936 a^3 b^2 \\
      & + 1704 a b^3) u^8 + (7 a^9 b^3 d^7 + 392 a^7 b^4 d^7 - 6863 a^5 b^5 d^7 - 17458 a^3 b^6 d^7 \\
      & - 515 a b^7 d^7 + 21 a^{10} b^2 d^6 + 1218 a^8 b^3 d^6 - 20647 a^6 b^4 d^6 - 61745 a^4 b^5 d^6 \\
      & - 6709 a^2 b^6 d^6 - 4 b^7 d^6 + 21 a^{11} b d^5 + 1302 a^9 b^2 d^5 - 20664 a^7 b^3 d^5 \\
      & - 81924 a^5 b^4 d^5 - 22146 a^3 b^5 d^5 - 183 a b^6 d^5 + 7 a^{12} d^4 + 567 a^{10} b d^4 \\
      & - 3982 a^8 b^2 d^4 - 97733 a^6 b^3 d^4 - 158644 a^4 b^4 d^4 - 8392 a^2 b^5 d^4 - 4 b^6 d^4 \\
      & + 84 a^{11} d^3 + 2878 a^9 b d^3 - 57242 a^7 b^2 d^3 - 160981 a^5 b^3 d^3 + 59447 a^3 b^4 d^3 \\
      & + 3223 a b^5 d^3 + 442 a^{10} d^2 + 2563 a^8 b d^2 - 142138 a^6 b^2 d^2 - 189134 a^4 b^3 d^2 \\
      & + 18323 a^2 b^4 d^2 + 20 b^5 d^2 + 885 a^9 d - 2382 a^7 b d - 179958 a^5 b^2 d \\
      & - 164688 a^3 b^3 d - 2637 a b^4 d + 696 a^8 - 5400 a^6 b - 92938 a^4 b^2 - 29078 a^2 b^3 \\
      & - 12 b^4) u^9 \big).  
\end{split}
\end{equation*}


\newcommand{\MRn}[2]{\href{http://www.ams.org/mathscinet-getitem?mr=#1}{MR#1 #2}}

\end{document}